\documentclass[12pt]{article}
\usepackage[colorlinks=true,citecolor=black,linkcolor=black,urlcolor=blue]{hyperref}

\usepackage[utf8]{inputenc}
\usepackage[T1]{fontenc}
\usepackage{amsmath,amsfonts,amsthm}
\usepackage{graphicx}

\newtheorem{theorem}{Theorem}[section]

\newtheorem{lemma}[theorem]{Lemma}

\newcommand{\LL}{\textnormal{\textsf{L}}}
\DeclareRobustCommand{\JJ}{\reflectbox{\LL}}
\newcommand{\outerl}{\textsc{Grounded}-\LL}
\newcommand{\outerlj}{\textsc{Grounded}-\ensuremath{\{\LL,\JJ\}}}
\newcommand{\cC}{\mathcal{C}}

\newcommand{\cls}[1]{\textsc{#1}}

\begin{document}

\title{On grounded L-graphs and their relatives}

\author{Vít Jelínek\thanks{Supported by the project 16-01602Y of the Czech Science 
Foundation and by project Neuron Impuls of the Neuron Fund for 
Support of Science.}\\
\small Computer Science Institute, Faculty of Mathematics and Physics,\\[-0.8ex]
\small Charles University, Prague, Czechia\\[-0.8ex]
\small\tt jelinek@iuuk.mff.cuni.cz\\
\and
Martin Töpfer\thanks{This project has
received funding from the
European Union’s Horizon 2020 research and innovation programme under the
Marie Skłodowska-Curie
Grant Agreement No. 665385.
}\\
\small Institute of Science and Technology, Klosterneuburg, Austria\\[-0.8ex]
\small\tt mtopfer@gmail.com
}

\maketitle

\begin{abstract}
We consider the graph classes \outerl\ and \outerlj\ corresponding to graphs that admit an 
intersection representation by \LL-shaped curves (or \LL-shaped and \JJ-shaped curves, 
respectively), where additionally the topmost points of each curve are assumed to belong to a 
common horizontal line. We prove that \outerl\ graphs admit an equivalent characterisation in terms 
of vertex ordering with forbidden patterns. 

We also compare these classes to related intersection classes, such as the grounded segment graphs, 
the monotone \LL-graphs (a.k.a. max point-tolerance graphs), or the outer-1-string graphs. We give 
constructions showing that these classes are all distinct and satisfy only trivial or previously 
known inclusions.

\end{abstract}
\section{Introduction}

An \emph{intersection representation} of a graph $G=(V,E)$ is a map that assigns to every vertex 
$x\in V$ a set $s_x$ in such a way that two vertices $x$ and $y$ are adjacent if and only if the two 
corresponding sets $s_x$ and $s_y$ intersect. The graph $G$ is then the \emph{intersection graph} of 
the set system $\{s_x;\; x\in V\}$. Many natural graph classes can be defined as intersection graphs 
of sets of a special type. 

One of the most general classes of this type is the class of \emph{string graphs}, denoted 
\cls{String}. A string graph is an intersection graph of \emph{strings}, which are bounded 
continuous curves in the plane. All the graph classes we consider in this paper are subclasses 
of string graphs. 

A natural way of restricting a string representation is to impose geometric restrictions on the 
strings we consider. This leads, for instance, to \emph{segment graphs}, which are intersection 
graphs of straight line segments, or to \emph{\LL-graphs}, which are intersection graphs of 
\LL-shapes, where an \LL-shape is a union of a vertical segment and a horizontal segment, in which 
the bottom endpoint of the vertical segment coincides with the left endpoint of the horizontal one.
Apart from \LL-shapes, we shall also consider \JJ-shapes, which are obtained by reflecting an 
\LL-shape along a vertical axis.

Apart from restricting the geometry of the strings, one may also restrict a string representation by 
imposing conditions on the placement of their endpoints. Following the terminology of Cardinal et 
al.~\cite{Cardinal}, we will say that a representation is \emph{grounded} if all the strings have 
one endpoint on a common line (called \emph{grounding line}) and the remaining points of the strings 
are confined to a single open halfplane with respect to the grounding line. We will usually assume 
that the grounding line is the $x$-axis, and the strings extend below the line. The endpoint 
belonging to the grounding line is the \emph{anchor} of the string.

Similarly, a string representation is an \emph{outer} representation, if all the strings are 
confined to a disk, and each string has one endpoint on the boundary of the disk. The endpoint on 
the boundary is again called the \emph{anchor} of the string. One may easily see that a graph admits 
a grounded string representation if and only if it admits an outer-string representation. Such 
graphs are known as \emph{outer-string} graphs, and we denote their class \cls{Outer-string}.

\begin{figure}[ht]
\centerline{\includegraphics[width=0.9\textwidth]{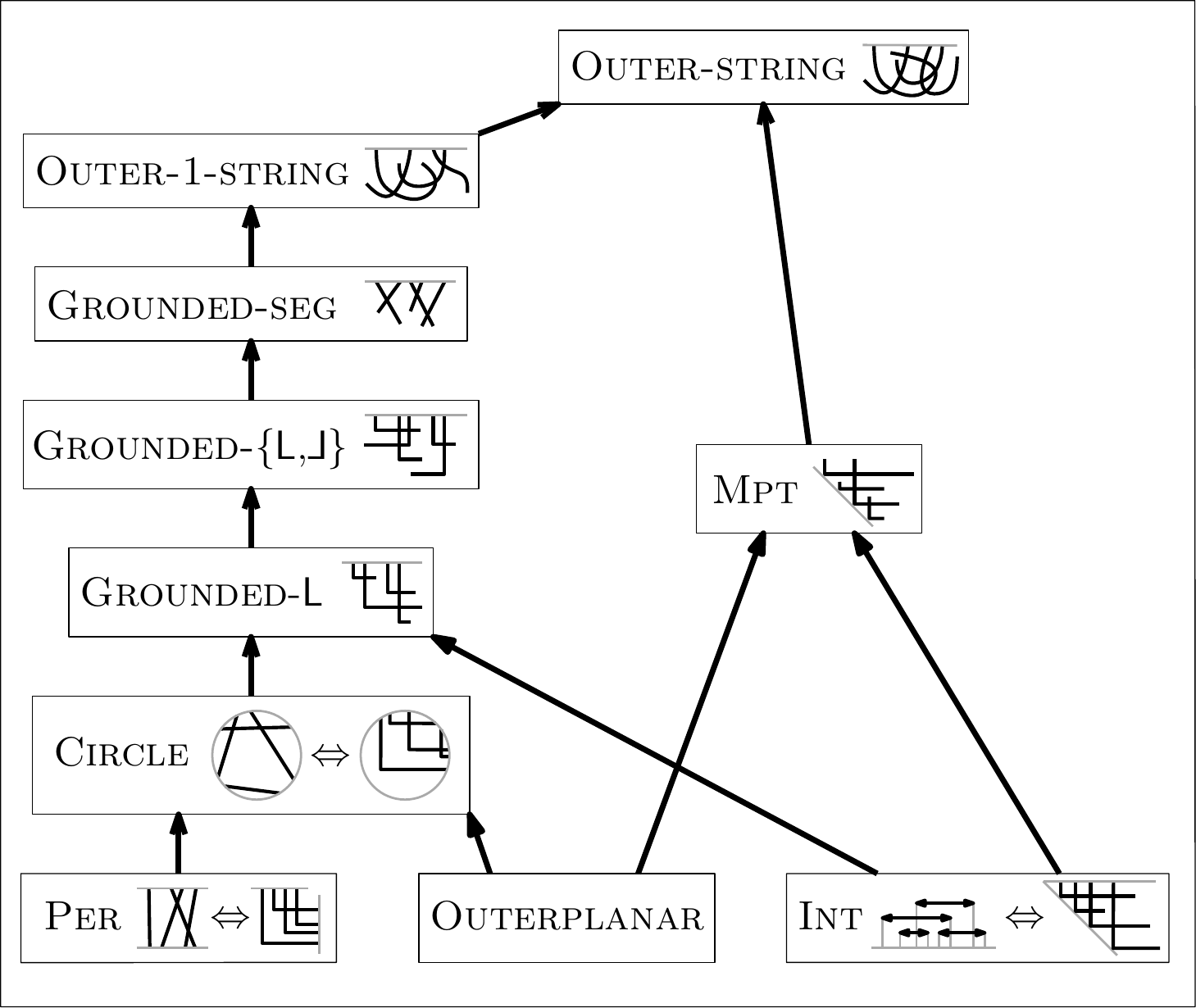}}
 \caption{Graph classes considered in this paper. Arrows indicate inclusions. We will see in 
Section~\ref{sec-sep} that there are no other inclusions among these classes apart from those 
implied by the depicted arrows. In particular, the classes are all distinct.}\label{fig-classes}
\end{figure}

Our first main result, Theorem~\ref{thm-patterns} in Section~\ref{sec-ord}, is a characterisation 
of the class of grounded \LL-graphs by vertex orderings avoiding a pair of forbidden patterns. Our 
next main result, presented in Section~\ref{sec-sep}, is a collection of constructions providing 
separations between the classes in Figure~\ref{fig-classes}, showing that there are no nontrivial 
previously unknown inclusions among them. 

Let us now formally introduce the graph classes we are interested in, and briefly review some 
relevant previously known results.

\emph{1-string graphs} are the graphs that admit a string representation in which any two distinct 
strings intersect at most once. The class of 1-string graphs is denoted \cls{1-String}.

\emph{Outer-1-string graphs} (denoted \cls{Outer-1-string}) are the graphs that have a string 
intersection 
representation which is simultaneously a 1-string representation and an outer-string 
representation. Note that not every graph from \cls{1-string}$\,\cap\,$\cls{Outer-string} is 
necessarily in \cls{Outer-1-string}, as we shall see in Section~\ref{sec-sep}.

\emph{\LL-graphs} (\LL) are the intersection graphs of \LL-shapes. This type of representation has 
received  significant amount of interest lately. A notable recent result is a theorem of Gon\c 
calves, Isenmann and Pennarun~\cite{Goncalves} showing that every planar graph is an \LL-graph. 
Since it is known that \LL-graphs are a subclass of segment graphs~\cite{Middendorf}, this result 
strengthens an earlier result of Chalopin and Gon\c calves~\cite{CG} showing that all planar graphs 
are segment graphs.

\emph{Max point-tolerance graphs} (\cls{Mpt}), also known as \emph{monotone \LL-graphs}, are the 
graphs with an \LL-representation in which all the bends of the \LL-shapes belong to a common 
downward-sloping line. This class was independently introduced by Soto and Thraves Caro~\cite{Soto}, 
by Catanzaro et al.~\cite{Catanzaro} and by Ahmed et al.~\cite{Ahmed}. Apart from the above 
intersection representation by \LL-shapes, it admits several other equivalent characterisations. 
Notably, \cls{Mpt} graphs can be characterised as graphs that admit a vertex ordering that avoids a 
certain forbidden pattern~\cite{Ahmed,Catanzaro,Soto}. This graph class is known to be a superclass 
of several important graph classes, such as outerplanar graphs and interval graphs, among 
others~\cite{Ahmed,Catanzaro,Soto}.

\emph{Grounded segment graphs} (\cls{Grounded-seg}) are the intersection graphs admitting a 
grounded segment representation. Cardinal et al.~\cite{Cardinal} proved that these are also precisely 
the intersection graphs of downward rays in the plane. Note that any grounded segment graph also 
admits an outer-segment representation, but the converse does not hold, as shown by Cardinal et 
al.~\cite{Cardinal}. Cardinal et al. also showed that outer-segment graphs are a proper subclass of 
outer-1-string graphs. This strengthens an earlier result of Cabello and 
Jejčič~\cite{CabelloJejcic}, who showed that outer-segment graphs are a proper subclass of 
outer-string graphs.

\emph{Grounded \LL-graphs} (\outerl) are the intersection graphs of groun\-ded \LL-shapes, that is, 
\LL-shapes with top endpoint on the $x$-axis. This class of graphs has been first considered by 
McGuiness~\cite{McGuiness}, who represented them as intersections of upward-infinite \LL-shapes.
These graphs can also equivalently be represented as intersections of \LL-shapes inside a disk, 
with the top endpoint of each \LL-shape anchored to the boundary of the disk. McGuiness has shown 
that this class is $\chi$-bounded, i.e., these graphs have chromatic number bounded from above by a 
function of their clique number. The $\chi$-boundedness result has been later generalized to all 
outer-string graphs by Rok and Walczak~\cite{Rok}.

\emph{Grounded $\{\LL,\JJ\}$-graphs} (\outerlj) are analogous to grounded \LL-graphs, but their 
representation may use both \LL-shapes and \JJ-shapes. An argument of Middendorf and 
Pfeiffer~\cite{Middendorf} shows that \outerlj\ is a subclass of \cls{Grounded-Seg}.

\emph{Circle graphs} (\cls{Circle}) are the intersection graphs of chords inside a circle, or 
equivalently, the intersection graphs of \LL-shapes drawn inside a circle, so that both endpoints 
of each \LL-shape touch the circle. Circle graphs include all outerplanar graphs~\cite{Wessel}.

\emph{Interval graphs} (\cls{Int}) are the intersection graphs of intervals on the real line. 
Equivalently, we may easily observe that these are exactly the graphs with an intersection 
representation which is simultaneously an \cls{Mpt}-repre\-sen\-tation and a 
\outerl-representation. But 
note that not every graph from the intersection of \cls{Mpt} and \outerl\ is an interval graph, as 
witnessed, e.g., by 
any cycle $C_n$ of length $n\ge 4$.

\emph{Permutation graphs} (\cls{Per}) are the intersection graphs of segments between a pair of 
parallel lines, with each segment having one endpoint on each of the two lines. Equivalently, we may 
observe that these are exactly the graphs admitting an \LL-representation in which the top endpoints 
of all the \LL-shapes are on a common horizontal line and the right endpoints are on a common 
vertical line.

We will always assume implicitly that the intersection representations we deal with satisfy certain 
non-degeneracy assumptions. In particular, we will assume that the strings have no 
self-intersections, that any two strings intersect in at most finitely many points (except for 
interval representations), and that any intersection of two strings is a proper crossing. In 
particular, an endpoint of a string does not belong to another string. Moreover, we will assume that 
every segment in a segment representation is non-degenerate, i.e., it has distinct endpoints. This 
also applies to horizontal and vertical segments forming an \LL-shape or \JJ-shape. These 
assumptions imply, in particular, that in any $\{\LL,\JJ\}$-representation, each intersection is 
realized as a crossing of a horizontal segment with a vertical one.

Note that in any grounded representation with a horizontal grounding line, the left-to-right 
ordering of the anchors on the grounding line defines a linear order on the vertex set of the 
represented graph. We say that this linear order is \emph{induced} by the representation. Similarly, 
for an \cls{Mpt} representation, we can define the induced order by following the top-left to 
bottom-right order of the bends along their common supporting line. Induced vertex orders  
play an important part both in characterising graphs in a given class and in separating 
different classes.

\section{Vertex orders with forbidden patterns}\label{sec-ord}

\newcommand{\Pa}{P_1}
\newcommand{\Pb}{P_2}

Our main result is a characterisation of grounded \LL-graphs as graphs that admit vertex orderings 
avoiding a pair of four-vertex patterns. Let us begin by formalising the key notions.

An \emph{ordered graph} is a pair $(G,<)$, where $G=(V,E)$ is a graph and $<$ is a linear order 
on~$V$. A \emph{pattern of order $k$} is a triple $P=(W,C,F)$ where $W$ is the set 
$\{1,2,\dotsc,k\}$
while $C$ and $F$ are two disjoint subsets of $\binom{W}{2}$. The set $W$ is the \emph{vertex set} 
of the pattern~$P$, $C$ is the set of \emph{compulsory edges} of~$P$, and $F$ is the set of 
\emph{forbidden edges}. 

For an ordered graph $(G,<)$  with $G=(V,E)$, we say that $(G,<)$ 
\emph{contains} a pattern $P=(W,C,F)$ of order $k$ if $G$ contains $k$ distinct vertices 
$x_1<x_2<\dotsb<x_k$ such that for every $\{i,j\}\in C$ the vertices $x_i$ and $x_j$ are adjacent 
in $G$, while for every $\{i,j\}\in F$, $x_i$ and $x_j$ are non-adjacent in~$G$. If $(G,<)$ does 
not contain $P$, we say that it \emph{avoids}~$P$. For simplicity, we will often write an edge 
$\{i,j\}$ as~$ij$.

\begin{figure}
\centerline{\includegraphics[width=\textwidth]{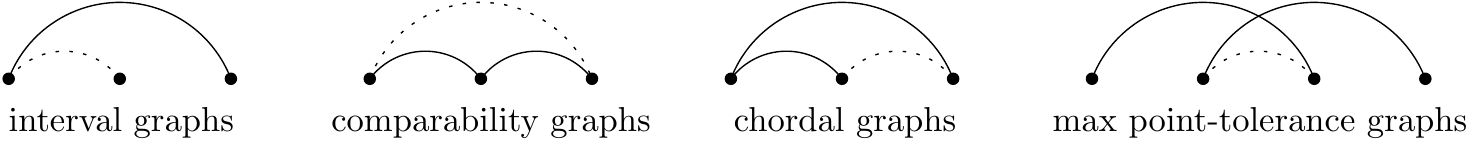}}
 \caption{Forbidden order patterns for various graph classes~\cite{Brandstadt,Catanzaro,Damaschke}. 
The solid arcs denote compulsory edges and the dotted arcs are forbidden edges.}\label{fig-pats}
\end{figure}

Many important graph classes can be characterised in terms of vertex orderings with forbidden 
patterns, that is, for a class $\cC$ there is a pattern $P_\cC$ such that a graph $G=(V,E)$ is in 
$\cC$ if and only if it admits a linear order $<$ such that $(G,<)$ avoids $P_\cC$; see 
Figure~\ref{fig-pats} for examples of classes with their forbidden patterns. The forbidden pattern
characterisation of \cls{Mpt} was found independently by at least three groups of 
authors~\cite{Ahmed,Catanzaro,Soto}. 

As our first main result, we show that \outerl\ is characterised by a pair of forbidden patterns.

\begin{figure}
\centerline{\includegraphics[width=0.8\textwidth]{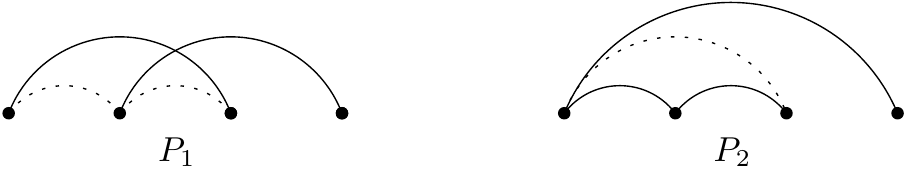}}
 \caption{The two forbidden ordering patterns for the class~\outerl.}\label{fig-grlpats}
\end{figure}

\begin{theorem}\label{thm-patterns}
Consider the patterns $\Pa=(\{1,2,3,4\},\{13,24\},\{12,23\})$ and 
$\Pb=(\{1,2,3,4\},\{12,14,23\},\{13\})$; see Figure~\ref{fig-grlpats}. A graph $G=(V,E)$ is a 
grounded \LL-graph if and only if it has a vertex ordering that avoids both $\Pa$ and~$\Pb$. In 
fact, a linear order $<$ on $V$ avoids the two patterns $\Pa$ and $\Pb$ if and only if $G$ has a 
grounded \LL-representation which induces the linear order~$<$.
\end{theorem}
\begin{proof}
Suppose first that $G$ has a grounded \LL-representation. Let $\ell_1,\ell_2,\dotsc,\ell_n$ be the 
\LL-shapes used in the representation, ordered left to right according to the positions of their 
anchors. Let $h_i$ and $v_i$ denote, respectively, the horizontal and vertical segment of~$\ell_i$. 
Let $x_i$ be the vertex represented by $\ell_i$. We will show that the vertex ordering 
$x_1<x_2<\dotsb<x_n$ avoids the two patterns $\Pa$ and~$\Pb$. 

Assume that $(G,<)$ contains $\Pa$, and let $x_p<x_q<x_r<x_s$ be the four vertices forming a copy 
of~$\Pa$. Since $x_q x_s$ is an edge, the two \LL-shapes $\ell_q$ and $\ell_s$ intersect. 
Let $R$ be the rectangle whose vertices are the anchors of $\ell_q$ and $\ell_s$, the bend of 
$\ell_q$ and the intersection of $\ell_q$ and~$\ell_s$. Since neither $x_p$ nor $x_r$ is adjacent 
to $x_q$, we see that $\ell_p$ is completely outside of $R$, while $v_r$ is inside~$R$. It follows 
that $\ell_p$ and $\ell_r$ are disjoint, and fail to represent the compulsory edge $13$ of~$\Pa$.

Suppose now that $(G,<)$ contains $\Pb$, and let $x_p<x_q<x_r<x_s$ now be the four vertices 
forming a copy~$\Pb$. Since $x_p x_s$ is an edge, the segment $h_p$ intersects~$v_s$. Moreover, 
$v_q$ intersects $h_p$, while $v_r$ does not intersect $h_p$, and in particular, $\ell_q$ and 
$\ell_r$ fail to represent the compulsory edge $23$ of~$\Pb$. We conclude that any grounded
\LL-representation of $G$ induces a vertex order that avoids $\Pa$ and~$\Pb$.

To prove the converse, assume that $G$ is a graph with a vertex ordering $x_1<x_2<\dotsb<x_n$ which 
avoids both $\Pa$ and~$\Pb$. We will construct a grounded \LL-representation 
$\ell_1,\ell_2,\dotsc,\ell_n$ of $G$ inducing the order $<$, with $\ell_i$ being the \LL-shape 
representing the vertex~$x_i$.

We fix the anchor of $\ell_i$ to be the point $(i,0)$ on the horizontal axis. Next, we process the 
vertices left to right, and for a vertex $x_i$ we define the representing shape $\ell_i$, 
assuming $\ell_1,\ell_2,\dotsc,\ell_{i-1}$ have already been defined, and assuming further that 
for any $j<i$ such that $x_j x_i$ is an edge of $G$, the horizontal segment $h_j$ of $\ell_j$ 
reaches 
to the right of the point~$(i,0)$. 

To define $\ell_i$, we first describe its vertical segment~$v_i$. Let $N^-_i$ be the set of vertices 
$x_j$ such that $j<i$ and $x_j x_i\in E$. If $N^-_i$ is empty, choose the vertical segment $v_i$ to 
be shorter than any of $v_1,\dotsc,v_{i-1}$. In particular, $v_i$ will not intersect any of the 
\LL-shapes constructed in previous steps. If $N^-_i$ is nonempty, let $x_p$ be a vertex from 
$N^-_i$ chosen so that $v_p$ is as long as possible (and therefore $h_p$ is as low as possible). 
Then define $v_i$ to be slightly longer than $v_p$, so that $v_i$ intersects $h_p$ (recall that 
$h_p$ reaches to the right of $(i,0)$) but does not intersect any \LL-shape whose horizontal segment 
is 
below~$h_p$. This choice of $v_i$ guarantees that $v_i$ intersects $h_j$ for any $x_j\in N^-_i$.

It remains to define the segment $h_i$. Let $j$ be the largest index such that $j>i$ 
and $x_i x_j\in E$. If no such $j$ exists, set $j=i$. The horizontal segment $h_i$ then has length 
$j-i+\frac12$, and in particular, its right endpoint has horizontal coordinate $j+\frac12$.

Having defined the \LL-shapes $\ell_1,\dotsc,\ell_n$ as above, let us verify that their 
intersection graph is~$G$. If $x_j x_i$ is an edge of $G$ with $j<i$, then the definition of $v_i$ 
guarantees that $v_i$ intersects $h_j$, and therefore the two \LL-shapes $\ell_j$ and $\ell_i$ 
intersect. 

To prove the converse, suppose for contradiction that for some $j<i$ the two  \LL-shapes $\ell_j$ 
and $\ell_i$ intersect while $x_j x_i$ is not an edge of~$G$. Choose such a pair $i,j$ so that $i$ 
is the smallest possible. There must be an index $k>i$ such that $x_j x_k$ is an edge of $G$, 
otherwise $h_j$ would be too short to intersect~$v_i$. Similarly, there must be an index $m<i$ such 
that $x_m x_i$ is an edge of $G$, and $v_m$ is longer than $v_j$, otherwise $v_i$ would not be long 
enough to intersect~$h_j$. 

We now distinguish two cases depending on the relative position of $m$ and~$j$. If $m<j$, then 
$\ell_m$ and $\ell_j$ are disjoint (recall that $v_m$ is longer than $v_j$) and hence $x_m x_j$ is 
not an edge of~$G$. It follows that the four vertices $x_m<x_j<x_i<x_k$ form the pattern $\Pa$, a 
contradiction. Suppose now that $j<m$. It follows that $\ell_j$ intersects $\ell_m$, and therefore 
$x_j x_m$ is an edge of $G$, by the minimality of~$i$. Thus, the four vertices $x_j<x_m<x_i<x_k$ 
form the pattern~$\Pb$, which is again a contradiction.
\end{proof}

\section{Separations between classes}\label{sec-sep}
Consider again the classes in Figure~\ref{fig-classes}. The inclusions indicated by arrows are 
either easy to observe or follow from known results that we have pointed out in the introduction. 
Our goal now is to argue that there are no other inclusions among these classes except those 
that follow by transitivity from the depicted arrows. In particular, the classes are all
distinct. 

\begin{figure}
\centerline{\includegraphics[width=0.9\textwidth]{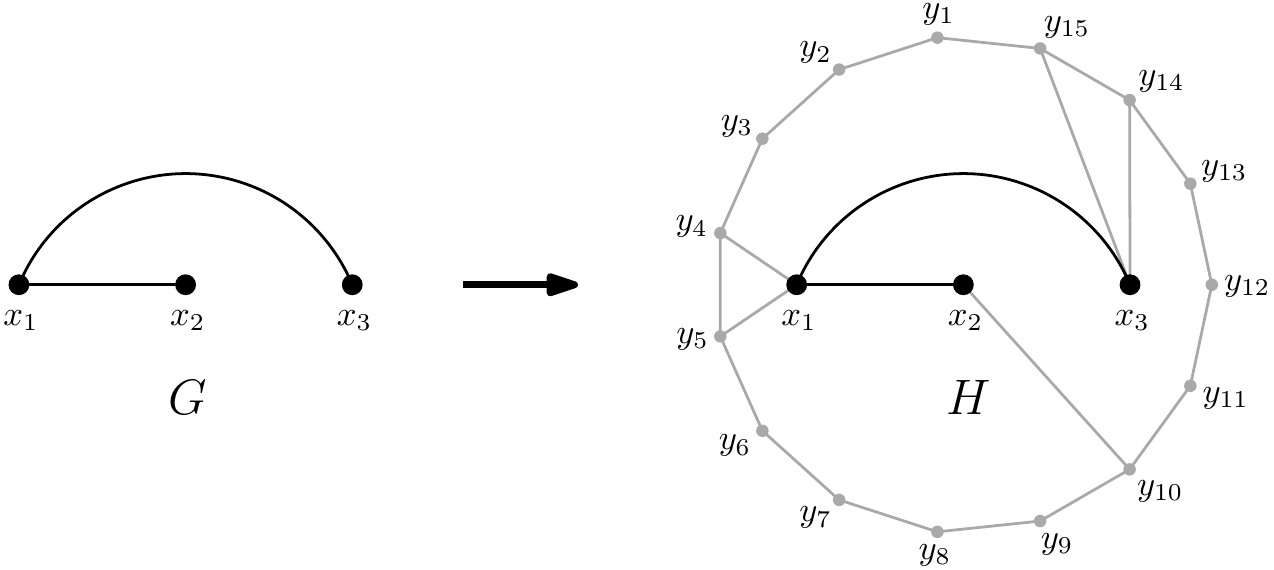}}
\caption{An ordered graph $G$, and an example of its cycle extension $H$.}\label{fig-cyclex}
\end{figure}

As our main tool, we will use a lemma which is a slight modification of the `Cycle Lemma' of 
Cardinal et al.~\cite{Cardinal}. The lemma allows to prescribe the cyclic order of a subset of 
vertices in an outer-1-string representation of a graph. Let $G=(V_G,E_G)$ be a graph on $n$ 
vertices $x_1, x_2,\dotsc,x_n$, and let $<$ be the linear order $x_1<x_2<\dotsb<x_n$. The
\emph{cyclic shift} of $<$ is the linear order $<_s$ defined as 
$x_n<_s x_1<_s x_2<_s\dotsb<_s x_{n-1}$. The \emph{reversal} of $<$, denoted $<_r$, is defined as
with $x_n <_r x_{n-1}<_r \dotsb<_r x_1$. We say that two linear orders of $V$ are \emph{equivalent} 
if one can be obtained from the other by a sequence of cyclic shifts and reversals.

A \emph{cycle extension} of the ordered graph $(G,<)$ is an (unordered) graph $H=(V_H,E_H)$ with 
these properties (see Figure~\ref{fig-cyclex}):
\begin{itemize}
\item $V_H$ is the disjoint union of the sets $V_G=\{x_1,\dotsc,x_n\}$ and 
$V_C=\{y_1,\dotsc,y_{5n}\}$. The vertices $V_G$ induce the graph $G$ (in particular, 
$E_G\subseteq E_H$), and $V_C$ induce a cycle of length $5n$ with edges $y_1 y_2,\allowbreak y_2 
y_3,\dotsc,\allowbreak 
y_{5n-1}y_{5n},\allowbreak y_{5n} y_1$.
\item For each vertex $x_i\in V_G$, either $x_i$ is adjacent to $y_{5i}$ and has no other neighbors 
in $V_C$, or $x_i$ is adjacent to $y_{5i-1}$ and $y_{5i}$ and has no other neighbors in~$V_C$.
\end{itemize}

For the classes of graphs we consider, an intersection representation of a graph $G$ inducing an 
order $<$ can always be extended into a representation of a cycle extension of $G$, without 
modifying the curves representing~$G$. This is formalised by the next lemma.

\begin{lemma}\label{lem-extend}
Given a graph class $\cC\in\{\text{\outerl},\text{\outerlj},\allowbreak\text{\cls{Mpt}},\allowbreak
\text{\cls{Grounded-seg}}, \allowbreak \text{\cls{Outer-1-string}}\}$, for every 
$\cC$-representation of a graph $G$ inducing an order $<$ on $V_G$ there is a cycle extension $H$ 
of~$(G,<)$ such that a $\cC$-representation of $H$ can be constructed by adding into the given 
$\cC$-representation of $G$ the curves representing the vertices of $V_H\setminus V_G$. 
\end{lemma}
\begin{proof}
Suppose we are given a $\cC$-representation of $G$. It is easy to see that we can add the curves 
representing the cycle $V_C$ close enough to the grounding line; see Figure~\ref{fig-extend}. Note 
that for \cls{Mpt}-representations, each original \LL-shape may have to be intersected by two 
consecutive \LL-shapes from the added cycle. In all the other types of representations, each vertex 
$x_i$ of $G$ will have a unique neighbor $y_{5i}$ among the $V_C$.
\end{proof}

\begin{figure}
\centerline{\includegraphics[width=0.9\textwidth]{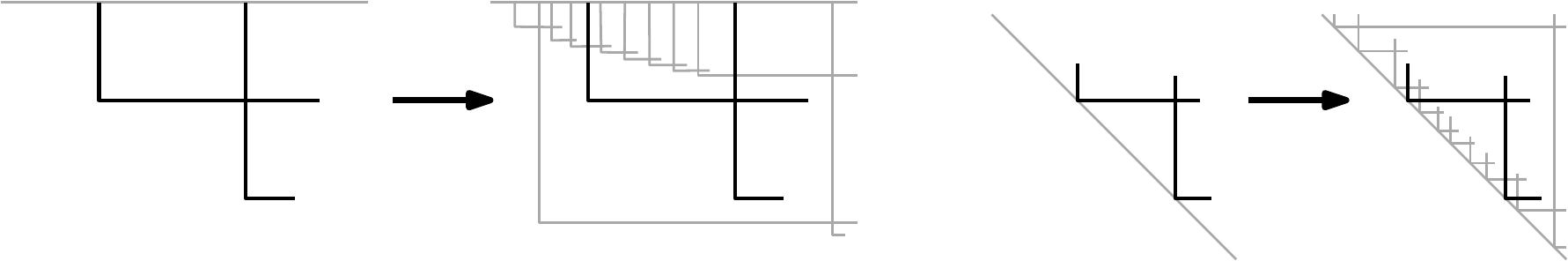}}
\caption{Extending the representation of $G$ into a representation of a cycle 
extension for grounded \LL-representations (left) and \cls{Mpt} representations 
(right).}\label{fig-extend}
\end{figure}

Recall that two linear orders are equivalent if one can be obtained from the other by a sequence of 
cyclic shifts and reversals. The key property of cycle extensions of $(G,<)$ is that they restrict 
the possible vertex orders of the $G$-part to an order equivalent to~$<$, as shown by the next 
lemma.

\begin{lemma}\label{lem-cycle2} If $(G,<)$ is an ordered graph with a cycle extension $H$, then in 
every grounded 1-string representation of $H$, the order of the vertices of $G$ induced by the 
representation is equivalent to the order~$<$.
\end{lemma}

\newcommand{\cnt}{\text{center}}
\newcommand{\strt}{\text{start}}

\begin{proof}
The proof follows the same ideas as the proof of the Cycle Lemma of Cardinal et al.~\cite[Lemma 
6]{Cardinal_arx}. 

Suppose $(G,<)$ is an ordered graph with vertices $V_G=\{x_1<x_2<\dotsb<x_n\}$ and edge-set $E_G$, 
and $H$ is its cycle extension, with vertices $V_H=V_G\cup V_C$ as in the definition of cycle 
extension and $V_C=\{y_1,\dotsc,y_{5n}\}$. When working with the indices of the vertices in $V_C$, 
we will assume that arithmetic operations are performed modulo~$5n$, so $5n+1=1$, etc.

Suppose that $H$ has a grounded 1-string representation. We may transform this representation into 
an outer-1-string representation, while preserving the induced vertex order up to equivalence. 
Suppose then that an outer-1-string representation of $H$ is given, inside a disk whose boundary is 
a circle~$B$. Let $c_j$ be the string 
representing~$y_j$, and let $p_{j,j+1}$ be the intersection point of $c_j$ and~$c_{j+1}$. The 
subcurve of $c_j$ between the two intersection points $p_{j-1,j}$ and $p_{j,j+1}$ is the 
\emph{central part} of~$c_j$, denoted~$\cnt(j)$. The part of $c_j$ between the anchor and the first 
point of $\cnt(j)$ is the \emph{initial part} of $c_j$, denoted $\strt(j)$. Let $p_j$ be the common 
endpoint of $\strt(j)$ and $\cnt(j)$. Note that $p_j$ is equal to $p_{j-1,j}$ or to $p_{j,j+1}$.
The sequence of curves $\cnt(1),\cnt(2),\dotsc,\cnt({5n})$ forms a closed Jordan curve, denoted 
by~$C$. Note that $C$ contains all the points $\{p_{k,k+1};\;k=1,\dotsc,5n\}$. Let $R_C$ be the 
interior region of~$C$.

Consider now a vertex $x_i$, represented by a string~$s_i$. Note that $s_i$ can only intersect the 
curve $C$ in a point of $\cnt({5i})$ or possibly~$\cnt({5i-1})$. Let $R_i$ be the planar region 
bounded by the union of the following four curves: $\strt(5i-3)$, $\strt(5i+2)$, the arc of $C$ 
between $p_{5i-3}$ and $p_{5i+2}$ that contains $\cnt(5i-1)\cup\cnt(5i)$, and the arc of $B$ between 
the anchors of $c_{5i-3}$ and $c_{5i+2}$ that contains the anchors of $c_{5i-1}$ and $c_{5i}$. 

Note that $s_i$ is the only string among the strings representing $V_G$ that can intersect the 
boundary of~$R_i$. Note also that the string $c_{5i}$ cannot intersect the boundary of $R_k$ for 
$k\neq i$, and therefore $c_{5i}$ is contained in $R_i\cup R_C$. Since $s_i$ intersects $c_{5i}$, and 
since $s_i$ also cannot cross the boundary of $R_k$ for $k\neq i$, it follows that $s_i$ is also 
contained in $R_C\cup R_i$, and in particular, the anchor of $s_i$ is in $R_i\cap B$. Therefore, the 
anchors of $s_1,\dotsc,s_n$ appear on $B$ in the order which, up to equivalence, corresponds to the 
order $<$ on~$V_G$.
\end{proof}

We will now use Lemmas~\ref{lem-extend} and~\ref{lem-cycle2} to construct graphs that have no 
representation in a given 
intersection class. Our goal is to show that there are no inclusions missing in 
Figure~\ref{fig-classes}. The classes \cls{Int}, \cls{Circle}, \cls{Outerplanar} and \cls{Per} are 
well studied~\cite{Brandstadt}, and simple examples show that there are no inclusions among them 
other than those depicted in Figure~\ref{fig-classes}. 

Catanzaro et al.~\cite[Observation 6.9]{Catanzaro} observed that the graph $K_{2,2,2}$ (the 
octahedron) is a permutation graph not in \cls{Mpt}, and therefore neither \cls{Per} nor 
any superclass of \cls{Per} is contained in \cls{Mpt}. Cardinal et al.~\cite{Cardinal} showed 
that \cls{Grounded-seg} is a proper subclass of \cls{Outer-1-string}. To complete the hierarchy, 
we only need the following separations.

\begin{figure}
\centerline{\includegraphics[width=0.9\textwidth]{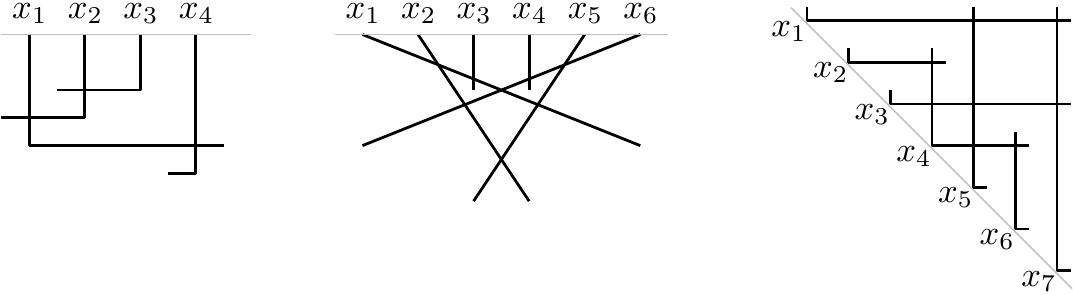}}
\caption{The three intersection representations used to prove 
Theorem~\ref{thm-separate}. In each case, a representation cannot be replaced by a representation from 
a smaller class while preserving the induced vertex order.
Left: a \outerlj\ representation which cannot be replaced by a \outerl\ one. Middle: 
a \cls{Grounded-seg} representation which cannot be replaced by a \outerlj\ one. 
Right: an \cls{Mpt} representation which cannot be replaced by an 
\cls{Outer-1-string} one.}\label{fig-sep}
\end{figure}

\begin{theorem}\label{thm-separate} The following properties hold.
\begin{itemize}
\item[(i)] The class \outerlj\ is not a subclass of \outerl.
\item[(ii)] The class \cls{Grounded-seg} is not a subclass of \outerlj.
\item[(iii)] The class \cls{Mpt} is not a subclass of~\cls{Outer-1-string}.
\end{itemize}
%
\end{theorem}
\begin{proof}
We first prove part (i) of the theorem. Consider the graph $G=(V,E)$ with 
$V=\{x_1,x_2,x_3,x_4\}$ and $E=\{x_1x_2,\allowbreak x_2x_3,\allowbreak x_1x_4\}$. 
Figure~\ref{fig-sep} (left) shows a grounded $\{\LL,\JJ\}$-representation of $G$ which induces the 
order $<$ defined as $x_1< x_2< x_3< x_4$ on $V$. Note that there is no grounded \LL-representation 
of $G$ that would induce the vertex order $<$, because $(G,<)$ contains the pattern $\Pb$ of 
Theorem~\ref{thm-patterns}. 

Let $(G', <')$ be the ordered graph obtained by putting $(G,<)$ and the mirror image of $(G,<)$ side 
by side. Formally, $(G',<')$ has vertex set $V'=\{x_1, x_2,x_3,\allowbreak x_4,y_1,y_2,\allowbreak 
y_3,y_4\}$, edge set $E'=\{ x_1x_2, x_2x_3, x_1x_4, y_1y_2, y_2y_3,y_1y_4\}$ and vertex order $x_1<' 
x_2<' x_3<' x_4<'y_4<'y_3<' y_2<' y_1$. Finally, let $(G'',<'')$ be the ordered graph obtained by 
placing two disjoint copies of $(G',<')$ side by side. Clearly $G''$ has a grounded
$\{\LL,\JJ\}$-representation which induces the vertex order $<''$. However, $G''$ has no grounded
\LL-representation inducing a vertex order equivalent with $<''$, since in any vertex order 
equivalent with $<''$ there are four consecutive vertices forming a copy of~$\Pb$.

By Lemma~\ref{lem-extend}, the ordered graph $(G'',<'')$ has a cycle extension 
$H$ that admits a grounded $\{\LL,\JJ\}$-representation. By Lemma~\ref{lem-cycle2}, any grounded 
1-string representation (and therefore any grounded \LL-representation) of $H$ induces on $V''$ an 
order which is equivalent with~$<''$. It follows that $H$ has no grounded \LL-representation, and 
therefore \outerlj\ is not a subclass of \outerl, as claimed.

For the other two parts of the theorem, the argument is analogous, the main difference is in the 
choice of the initial ordered graph~$(G,<)$. To prove part (ii), consider 
the graph $G$ on six vertices whose \cls{Grounded-seg} representation is in the middle of 
Figure~\ref{fig-sep}, and let $<$ be the vertex order induced by the depicted representation. 

Let us argue that $G$ has no grounded $\{\LL,\JJ\}$-representation inducing the vertex order~$<$. 
For contradiction, suppose that such a representation exists, and let $\ell_i$ denote the \LL-shape 
or \JJ-shape representing~$x_i$ in this representation. Let $h_i$ and $v_i$ be the horizontal and 
vertical segment of~$\ell_i$, respectively. Suppose, without loss of generality, that $v_1$ is 
longer than~$v_6$. Since $\ell_1$ and $\ell_6$ intersect, $h_6$ must intersect $v_1$, and $\ell_6$ 
is a \JJ-shape. Since $\ell_2$ intersects both $\ell_1$ and $\ell_6$, $v_2$ must be longer than 
$v_6$, and $v_2$ intersects~$h_6$. But this means that $\ell_3$ must intersect either 
$\ell_2$ or $\ell_6$ in order to intersect $\ell_1$, a contradiction.

Note that the graph $(G,<)$ is isomorphic to its reversal. Consider the ordered graph $(G',<')$ 
obtained by placing two copies of $(G,<)$ side by side: note that in any vertex order equivalent to 
$<'$, $G'$ contains a copy of $(G,<)$, and therefore there is no grounded 
$\{\LL,\JJ\}$-representation of $G'$ inducing a vertex order equivalent to~$<'$. We apply 
Lemmas~\ref{lem-extend} and \ref{lem-cycle2} to $(G',<')$ and obtain its cycle extension $H$, which 
is in \cls{Grounded-seg} but not in \outerlj.

To prove part (iii), consider the graph $G$ whose \cls{Mpt}-representation is depicted in the right 
part of Figure~\ref{fig-sep}, and let $<$ be the vertex order induced by the representation. We 
claim that there is no grounded 1-string representation of $G$ inducing the order~$<$. For 
contradiction, suppose that such a representation exists, and let $s_i$ be the string representing 
the vertex~$x_i$. Additionally, let $a_i$ denote the anchor of $s_i$, and for a pair of intersecting 
strings $s_i$, $s_j$ let $p_{ij}$ 
denote their intersection. 

Suppose, without loss of generality, that when we follow $s_4$ starting in $a_4$, we encounter 
$p_{24}$ before we encounter $p_{46}$. Let $C$ be the closed Jordan curve obtained as the union of 
the subcurve of $s_1$ between $a_1$ and $p_{17}$, the subcurve of $s_7$ between $p_{17}$ and 
$p_{37}$, the subcurve of $s_3$ between $p_{37}$ and $a_3$, and the segment $a_1a_3$ of the grounding 
line. Note that $s_2$ is inside $C$ (except $a_2$, which lies on $C$), and both $a_4$ and $s_6$ are 
outside~$C$. Therefore, $s_4$ must intersect $C$ at least twice: once between $a_4$ and $p_{24}$, 
and once between $p_{24}$ and $p_{46}$. However, $s_4$ can only intersect $C$ in the point $p_{34}$, 
a contradiction. 

To complete the proof, we first observe that $G$ has no grounded 1-string representation inducing a 
vertex order equivalent with~$<$, since such a representation could be trivially transformed into a 
grounded 1-string representation inducing~$<$. We apply Lemmas~\ref{lem-extend} and 
\ref{lem-cycle2} to $G$, to obtain a graph $H$ which is in $\cls{Mpt}$ but not in 
\cls{Outer-1-string}.
\end{proof}

Note that these results imply that \cls{Outer-string} is a proper superclass of both \cls{Mpt} 
and \cls{Outer-1-string}.

We remark that \cls{Mpt} is clearly a subclass of \cls{1-string} and of \cls{Outer-string}, but it 
is not a subclass of \cls{Outer-1-string}, as we just saw.

\section{Concluding remarks}

We have seen that the vertex orders induced by grounded \LL-representations can be characterised by 
a pair of forbidden patterns. Previously, a characterisation by a single forbidden pattern has been 
found for vertex orders induced by \cls{Mpt} representations~\cite{Ahmed,Catanzaro,Soto}. It is an 
open 
problem whether such a characterisation can be obtained for other similar grounded intersection 
classes, such as the class \outerlj.

Another problem concerns the recognition complexity of the graph classes we considered. Recognition 
of max point-tolerance graphs is mentioned as a prominent open problem by Ahmed et al.~\cite{Ahmed}, 
by Catanzaro et al.~\cite{Catanzaro}, as well as by Soto and Thraves Caro~\cite{Soto}. For the 
classes \outerl\ and \outerlj, recognition is open as well. On the other hand, the recognition 
problem for \cls{Grounded-seg} is known to be $\exists \mathbb{R}$-complete, as shown by Cardinal et 
al.~\cite{Cardinal}. In particular, \cls{Grounded-seg} cannot be characterised by finitely many 
forbidden vertex order patterns, unless $\exists\mathbb{R}$ is equal to NP.

The characterisation of \outerl\ by forbidden vertex order patterns might conceivably be helpful in 
designing a polynomial recognition algorithm, but note that even a graph class characterised by a 
forbidden vertex order pattern may have NP-hard recognition~\cite{Duffus}, although it is known that 
recognition is polynomial for all classes described by a set of forbidden patterns of order at most 
three~\cite{HMR}.

%
%
\bibliographystyle{plain}
\bibliography{outerl}

\begin{thebibliography}{10}

\bibitem{Ahmed}
Abu~Reyan Ahmed, Felice~De Luca, Sabin Devkota, Alon Efrat, Md.~Iqbal Hossain,
  Stephen~G. Kobourov, Jixian Li, Sammi~Abida Salma, and Eric Welch.
\newblock L-graphs and monotone {L}-graphs.
\newblock arXiv:1703.01544, 2017.

\bibitem{Brandstadt}
Andreas Brandst\"adt, Van~Bang Le, and Jeremy~P. Spinrad.
\newblock {\em Graph classes: a survey}.
\newblock SIAM Monographs on Discrete Mathematics and Applications. Society for
  Industrial and Applied Mathematics (SIAM), Philadelphia, PA, 1999.

\bibitem{CabelloJejcic}
Sergio Cabello and Miha Jej{\v c}i{\v c}.
\newblock Refining the hierarchies of classes of geometric intersection graphs.
\newblock {\em Electron. J. Combin.}, 24(1):Paper 1.33, 19, 2017.

\bibitem{Cardinal_arx}
Jean Cardinal, Stefan Felsner, Tillmann Miltzow, Casey Tompkins, and Birgit
  Vogtenhuber.
\newblock Intersection graphs of rays and grounded segments.
\newblock arXiv:1612.03638v1, 2016.

\bibitem{Cardinal}
Jean Cardinal, Stefan Felsner, Tillmann Miltzow, Casey Tompkins, and Birgit
  Vogtenhuber.
\newblock Intersection graphs of rays and grounded segments.
\newblock In {\em Graph-theoretic concepts in computer science}, volume 10520
  of {\em Lecture Notes in Comput. Sci.}, pages 153--166. Springer, Cham, 2017.

\bibitem{Catanzaro}
Daniele Catanzaro, Steven Chaplick, Stefan Felsner, Bjarni~V. Halld\'orsson,
  Magn\'us~M. Halld\'orsson, Thomas Hixon, and Juraj Stacho.
\newblock Max point-tolerance graphs.
\newblock {\em Discrete Appl. Math.}, 216(part 1):84--97, 2017.

\bibitem{CG}
J{\'e}r{\'e}mie Chalopin and Daniel Gon\c{c}alves.
\newblock Every planar graph is the intersection graph of segments in the
  plane: Extended abstract.
\newblock In {\em Proceedings of the Forty-first Annual ACM Symposium on Theory
  of Computing}, STOC '09, pages 631--638, New York, NY, USA, 2009. ACM.

\bibitem{Damaschke}
P.~Damaschke.
\newblock Forbidden ordered subgraphs.
\newblock In {\em Topics in combinatorics and graph theory ({O}berwolfach,
  1990)}, pages 219--229. Physica, Heidelberg, 1990.

\bibitem{Duffus}
Dwight Duffus, Mark Ginn, and Vojt{\v e}ch R\"odl.
\newblock On the computational complexity of ordered subgraph recognition.
\newblock {\em Random Structures Algorithms}, 7(3):223--268, 1995.

\bibitem{Goncalves}
Daniel Gon\c{c}alves, Lucas Isenmann, and Claire Pennarun.
\newblock Planar graphs as {L}-intersection or {L}-contact graphs.
\newblock In {\em Proceedings of the Twenty-Ninth Annual ACM-SIAM Symposium on
  Discrete Algorithms}, SODA '18, pages 172--184, Philadelphia, PA, USA, 2018.
  Society for Industrial and Applied Mathematics.

\bibitem{HMR}
Pavol Hell, Bojan Mohar, and Arash Rafiey.
\newblock Ordering without forbidden patterns.
\newblock In {\em Algorithms---{ESA} 2014}, volume 8737 of {\em Lecture Notes
  in Comput. Sci.}, pages 554--565. Springer, Heidelberg, 2014.

\bibitem{McGuiness}
Sean McGuinness.
\newblock On bounding the chromatic number of {L}-graphs.
\newblock {\em Discrete Math.}, 154(1-3):179--187, 1996.

\bibitem{Middendorf}
M.~Middendorf and F.~Pfeiffer.
\newblock The max clique problem in classes of string-graphs.
\newblock {\em Discrete Math.}, 108(1-3):365--372, 1992.

\bibitem{Rok}
Alexandre Rok and Bartosz Walczak.
\newblock Outerstring graphs are {$\chi$}-bounded.
\newblock In {\em Computational geometry ({S}o{CG}'14)}, pages 136--143. ACM,
  New York, 2014.

\bibitem{Soto}
Mauricio Soto and Christopher Thraves~Caro.
\newblock {$p$}-box: a new graph model.
\newblock {\em Discrete Math. Theor. Comput. Sci.}, 17(1):169--186, 2015.

\bibitem{Wessel}
Walter Wessel and Reinhard P{\"o}schel.
\newblock On circle graphs.
\newblock In {\em Graphs, hypergraphs and applications ({E}yba, 1984)},
  volume~73 of {\em Teubner-Texte Math.}, pages 207--210. Teubner, Leipzig,
  1985.

\end{thebibliography}
%
%
%
%
%
\end{document}